\documentclass[10pt, article]{amsart}

\usepackage{tikz}
\usetikzlibrary{calc}
\usepackage{ae} 
\usepackage[T1]{fontenc}
\usepackage[cp1250]{inputenc}
\usepackage{amsmath}
\usepackage{amssymb, amsfonts,amscd,verbatim}

\usepackage[normalem]{ulem}
\usepackage{hyperref}
\usepackage{indentfirst}
\usepackage{latexsym}
\input xy
\xyoption{all}

\usepackage{amsmath}    

\theoremstyle{plain}
\newtheorem{Pocz}{Poczatek}[section]
\newtheorem{Proposition}[Pocz]{Proposition}
\newtheorem{Theorem}[Pocz]{Theorem}
\newtheorem{Corollary}[Pocz]{Corollary}

\newtheorem{Observation}[Pocz]{Observation}

\newtheorem{Question}[Pocz]{Question}

\newtheorem{Example}[Pocz]{Example}

\theoremstyle{definition}
\newtheorem{Definition}[Pocz]{Definition}

\theoremstyle{remark}
\newtheorem{Remark}[Pocz]{Remark}

\DeclareMathOperator*{\diam}{diam}

\def\diam{\mathrm{diam}}

\numberwithin{equation}{section}

\title[
Ends and simple coarse structures
]%
  {Ends and simple coarse structures}

\author{Jerzy ~Dydak}
\address{University of Tennessee, Knoxville, USA}
\email{jdydak@utk.edu}

\date{ \today
}


\begin{document}
\maketitle

\tableofcontents

\begin{abstract}
This paper is devoted to introducing coarse structures in a very simple way, namely as an equivalence relation on the set of simple ends. As an application we show that Gromov boundary of every hyperbolic space is an example of a Higson corona and each Freundenthal compactification is an example of a Higson compactification.
\end{abstract}

\section{Introduction}
Abstract coarse structures were introduced by J.Roe \cite{Roe lectures} (see \cite{YuNo} for another exposition of coarse theory). Subsequently, equivalent structures, called large scale structures were introduced by J.Dydak and C.Hoffland \cite{DH}. 
This paper is devoted to a much simpler definition of majority of useful coarse structures. Namely, they are equivalence relations on the set of simple ends of sets equipped with a bounded structure.

\section{Bounded structures and simple ends}

This section is devoted to the first step in defining simple coarse structures. Namely, in order to define simple ends we need the concept of a bounded structure. The concept of ends is well-developed in the literature (see \cite{Peschke}, for example). Notice that those ends are usually equivalence classes of our simple ends.

\begin{Definition}
A \textbf{bounded structure} $\mathcal{B}$ on a set $X$ is a family of subsets of $X$ satisfying the following conditions:\\
1. $\{x\}\in \mathcal{B}$ for each $x\in X$,\\
2. $A\in \mathcal{B}$ if there is $C\in \mathcal{B}$ containing $A$,\\
3. $A\cup C\in \mathcal{B}$ if $A,C\in \mathcal{B}$ and $A\cap C\ne \emptyset$.

Elements of $\mathcal{B}$ are called \textbf{bounded subsets} of $X$.
\end{Definition}

\begin{Observation}
 Notice the difference between the notion of a bounded structure and that of a \textbf{bornology} (see \cite{Beer} and \cite{Hu}).
 Namely, each bornology is closed under arbitrary finite unions and a bounded structure is closed
 under finite unions of bounded sets that have non-empty intersection.
\end{Observation}

\begin{Definition}\label{SimpleEndDef}
Suppose $(X,\mathcal{B})$ is a set $X$ equipped with a bounded structure $\mathcal{B}$.
A \textbf{simple end} in $(X,\mathcal{B})$ is a sequence $\{x_n\}_{n=1}^\infty$ in $X$ with the property that for any bounded set $A$ the set $\{n\in \mathbb{N} \mid x_n\in A\}$ is finite.
\end{Definition}

\begin{Example}
Any $\infty$-pseudo-metric space $(X,d)$ induces the bounded structure $\mathcal{B}_d$
consisting of subsets of $X$ of finite diameter. If $d$ is a metric, then a sequence $\{x_n\}_{n=1}^\infty$ in $X$ is a simple end if and only if $\lim\limits_{n\to\infty}d(x_n,p)=\infty$ for some, and hence for all, $p\in X$.
\end{Example}

\begin{Example}
Any topological space $(X,\mathcal{T})$ induces the bounded structure $\mathcal{T}_c$
consisting of pre-compact subsets of $X$, i.e. of subsets whose closure in $X$ is compact.
If $X$ is locally compact, then a sequence $\{x_n\}_{n=1}^\infty$ in $X$ is a simple end if and only if $\lim\limits_{n\to\infty}x_n=\infty$ in the one-point compactification $X\cup \{\infty\}$ of $X$.
\end{Example}

\begin{Example}
Any group $G$ induces the bounded structure $G_f$
consisting of finite subsets of $G$. If $G$ is finitely generated and $d$ is a left-invariant word metric on $G$, then a sequence $\{x_n\}_{n=1}^\infty$ in $G$ is a simple end if and only if $\lim\limits_{n\to\infty}d(x_n,p)=\infty$ for some, and hence for all, $p\in G$.
\end{Example}

\begin{Remark}
It is tempting to define simple ends directly, bypassing bounded structures. That means creating axioms along the following lines: $\mathcal{SE}$ is a set of simple ends if it has the following properties:\\
1. Every infinite subsequence of $s\in   \mathcal{SE}$ belongs to $\mathcal{SE}$,\\
2. No constant sequence belongs to $\mathcal{SE}$.

Such $\mathcal{SE}$ would lead to a bounded structure $\mathcal{B}$: $B\in \mathcal{B}$
if and only if no sequence in $B$ belongs to $\mathcal{SE}$. However, $\mathcal{B}$ is always a bornology, so that leads to a less general concept.

\end{Remark}

\section{Simple coarse structures}
\begin{Definition}\label{SCSDefinition}
Suppose $(X,\mathcal{B})$ is a set $X$ equipped with a bounded structure $\mathcal{B}$.
A \textbf{simple coarse structure} in $X$ is an equivalence relation on the set of simple ends of $(X,\mathcal{B})$.
\end{Definition}

\begin{Example}\label{MetricSCS}
Any $\infty$-pseudo-metric space $(X,d)$ induces the simple coarse structure $\mathcal{SCS}_d$ as follows: \\
1. its bounded sets are are subsets of $X$ of finite diameter,\\ 
2. two simple ends $\{x_n\}_{n=1}^\infty$ and $\{y_n\}_{n=1}^\infty$ are equivalent if and only if there is $M > 0$ such that
$d(x_n,y_n) < M$ for all $n\ge 1$.
\end{Example}

\begin{Example}\label{C0SCS}
Any $\infty$-pseudo-metric space $(X,d)$ induces the simple coarse structure $\mathcal{SCS}(C_0)_d$ as follows: \\
1. its bounded sets are are subsets of $X$ of finite diameter,\\ 
2. two simple ends $\{x_n\}_{n=1}^\infty$ and $\{y_n\}_{n=1}^\infty$ are equivalent if and only if 
$d(x_n,y_n)\to 0$ as $n\to\infty$.
\end{Example}

\begin{Example}\label{MetricExtensionSCS}
Any metric extension $(\bar X,d)$ of a topological space $(X,\mathcal{T})$ (i.e. $X\subset \bar X$ and the topology $\mathcal{T}$ is identical with the one induced on $X$ by the metric $d$) induces the simple coarse structure $SCS_m(\bar X,X)$ on $X$ as follows:\\
1. its bounded sets are subsets $B$ of $X$ so that there is $\epsilon > 0$
with $dist(b,\bar X\setminus X) > \epsilon$ for all $b\in B$,\\
2. two simple ends $\{x_n\}_{n=1}^\infty$ and $\{y_n\}_{n=1}^\infty$ are equivalent if and only if $\lim\limits_{n\to\infty} d(x_n,y_n)=0$.
\end{Example}

\begin{Example}\label{ExtensionSCS}
Any extension $\bar X$ of a topological space $(X,\mathcal{T})$ (i.e. $X\subset \bar X$ and the topology $\mathcal{T}$ is identical with the one induced on $X$ from $\bar X$) induces the simple coarse structure $SCS(\bar X,X)$ on $X$ as follows: \\
1. its bounded sets are subsets of $X$ whose closure in $\bar X$ is contained in $X$,\\
2. two simple ends $\{x_n\}_{n=1}^\infty$ and $\{y_n\}_{n=1}^\infty$ are equivalent if and only if, for any function $a:\mathbb{N}\to\mathbb{N}$ with $\lim\limits_{n\to \infty}a(n)=\infty$,
the
coronas of closures in $\bar X$ of $\{x_{a(n)}\}_{n=1}^\infty$ and $\{y_{a(n)}\}_{n=1}^\infty$ are equal, i.e. $cl(\{x_{a(n)}\}_{n=1}^\infty)\setminus X= cl(\{y_{a(n)}\}_{n=1}^\infty)\setminus X$.
\end{Example}

\begin{Example}\label{SCSInducedBySubalgebras}
Suppose $(X,\mathcal{B})$ is a set equipped with a bounded structure and $\mathcal{P}$
is a family of functions from $X$ to $[0,1]$.
$\mathcal{P}$ induces the simple coarse structure $SCS_\mathcal{P}$ on $X$ as follows:\\
1. its bounded sets are the same as in $\mathcal{B}$,\\
2. two simple ends $\{x_n\}_{n=1}^\infty$ and $\{y_n\}_{n=1}^\infty$ are equivalent if and only if 
$$\lim\limits_{n\to\infty} |f(x_n)-f(y_n)|=0$$
 for any $f\in \mathcal{P}$.
\end{Example}

\begin{Remark} Obviously, \ref{SCSInducedBySubalgebras} can be generalized to any family of functions to, possibly different, metric spaces. However, maps to the unit interval are of most interest.
\end{Remark}

\begin{Example}\label{GroupSCS}
Any group $G$ has the simple coarse structure $SCS_l(G)$ as follows: two simple ends $\{x_n\}_{n=1}^\infty$ and $\{y_n\}_{n=1}^\infty$ are equivalent if and only if there is a finite subset $F$ of $G$ such that $x_n^{-1}\cdot y_n\in F$ for all $n\ge 1$.
\end{Example}

The above example can be generalized to the next one:
\begin{Example}\label{TopologicalGroupSCS}
Any topological group $G$ induces the simple coarse structure $SCS_l(G)$ on $G$ as follows: two simple ends $\{x_n\}_{n=1}^\infty$ and $\{y_n\}_{n=1}^\infty$ are equivalent if and only if there is a compact subset $F$ of $G$ such that $x_n^{-1}\cdot y_n\in F$ for all $n\ge 1$.
\end{Example}

Given a metric space $(X,d)$, the \textbf{Gromov product} of $x$ and $y$ with respect to $a\in X$ is defined by
$$
\left< x,y\right>_a=\frac{1}{2}\big(d(x,a)+d(y,a)-d(x,y)\big).
$$

Recall that metric space $(X,d)$ is (Gromov) $ \delta-$\textbf{hyperbolic} if it satisfies the $\delta/4$-inequality:
$$
\left< x,y\right>_{a} \geq \min \{\left< x,z\right>_{a},\left< z,y \right> _{a}\}-\delta/4, \quad \forall x,y,z,a\in X.$$

$(X,d)$ is \textbf{Gromov hyperbolic} if it is $ \delta-$hyperbolic for some $\delta > 0$.

\begin{Example}\label{HyperbolicSCS}
Any Gromov hyperbolic space $X$ induces the simple coarse structure $\mathcal{H}(X)$ on $X$ as follows:
two simple ends $\{x_n\}_{n=1}^\infty$ and $\{y_n\}_{n=1}^\infty$ are equivalent if and only if and only if $\left< x_n,y_n\right>_p\to\infty$ for some point $p\in X$.

\end{Example}

\begin{Theorem}
Suppose $X$ is a $\sigma$-compact locally compact Hausdorff space.
 If $\bar X$ is a compactification of $X$, then the following two simple coarse structures induced on $X$ and based on pre-compact subsets of $X$ are equal:\\
1. two simple ends $\{x_n\}_{n=1}^\infty$ and $\{y_n\}_{n=1}^\infty$ are equivalent if and only if, for any function $a:\mathbb{N}\to\mathbb{N}$ with $\lim\limits_{n\to \infty}a(n)=\infty$,
the
coronas of closures in $\bar X$ of $\{x_{a(n)}\}_{n=1}^\infty$ and $\{y_{a(n)}\}_{n=1}^\infty$ are equal, i.e. $cl(\{x_{a(n)}\}_{n=1}^\infty)\setminus X= cl(\{y_{a(n)}\}_{n=1}^\infty)\setminus X$.\\
2. two simple ends $\{x_n\}_{n=1}^\infty$ and $\{y_n\}_{n=1}^\infty$ are equivalent if and only if 
$$\lim\limits_{n\to\infty} |f(x_n)-f(y_n)|=0$$
 for every continuous function $f:\bar X\to [0,1]$.
\end{Theorem}
\begin{proof}
Consider two simple ends $\{x_n\}_{n=1}^\infty$ and $\{y_n\}_{n=1}^\infty$. If there is a function $a:\mathbb{N}\to\mathbb{N}$ with $\lim\limits_{n\to \infty}a(n)=\infty$ so that
the
coronas of closures in $\bar X$ of $\{x_{a(n)}\}_{n=1}^\infty$ and $\{y_{a(n)}\}_{n=1}^\infty$ are not equal, we may assume $ x_0$ is in the corona of $\{x_{a(n)}\}_{n=1}^\infty$  and does not belong to the corona $Y$ of $\{y_{a(n)}\}_{n=1}^\infty$. In that case there is a continuous function
$f:\bar X\to [0,1]$ so that $f(x_0)=0$ and $f(Y)=1$. Hence $\lim\limits_{n\to\infty} |f(x_n)-f(y_n)|=0$ is false.

Conversely, if two simple ends $\{x_n\}_{n=1}^\infty$ and $\{y_n\}_{n=1}^\infty$ satisfy 
$\lim\limits_{n\to\infty} |f(x_n)-f(y_n)|\ne 0$ (or it does not exist) 
 for some continuous function $f:\bar X\to [0,1]$, then we may switch to subsequences $\{x_{a(n)}\}_{n=1}^\infty$ and $\{y_{a(n)}\}_{n=1}^\infty$
so that $\lim\limits_{n\to\infty} |f(x_{a(n)})-f(y_{a(n)})|=c\ne 0$.
In that case the coronas are disjoint.
\end{proof}

\begin{Theorem}
Suppose $X$ is a $\sigma$-compact locally compact Hausdorff space.
 If $X$ has a metric compactification $(\bar X,d)$, then the following two simple coarse structures induced on $X$ and based on pre-compact subsets of $X$ are equal:\\
1. two simple ends $\{x_n\}_{n=1}^\infty$ and $\{y_n\}_{n=1}^\infty$ are equivalent if and only if $\lim\limits_{n\to\infty} d(x_n,y_n)=0$.\\
2. two simple ends $\{x_n\}_{n=1}^\infty$ and $\{y_n\}_{n=1}^\infty$ are equivalent if and only if 
$$\lim\limits_{n\to\infty} |f(x_n)-f(y_n)|=0$$
 for every continuous function $f:\bar X\to [0,1]$.
\end{Theorem}
\begin{proof}
Obviously, given two simple ends $\{x_n\}_{n=1}^\infty$ and $\{y_n\}_{n=1}^\infty$  satisfying $\lim\limits_{n\to\infty} d(x_n,y_n)=0$, one can see that $$\lim\limits_{n\to\infty} |f(x_n)-f(y_n)|=0$$
 for every continuous function $f:\bar X\to [0,1]$ due to uniform continuity of $f$.

Conversely, if two simple ends $\{x_n\}_{n=1}^\infty$ and $\{y_n\}_{n=1}^\infty$ satisfy $d(x_n,y_n)\ge\epsilon > 0$,
 for some $\epsilon$ and all $n\ge 1$, then the
coronas of closures in $\bar X$ of $\{x_{n}\}_{n=1}^\infty$ and $\{y_{n}\}_{n=1}^\infty$ are disjoint. In that case one can construct a continuous function $f$ on $\bar X$ that is equal $0$ on one corona and $1$ on the other.
\end{proof}

\section{Roe coarse structures and large scale structures}

For basic facts related to the coarse category see \cite{Roe lectures}.

Recall that a {\bf coarse structure} $\mathcal{C}$ on $X$ is a family of subsets $E$
(called {\bf controlled sets})
of $X\times X$ satisfying the following properties:
\begin{enumerate}
\item The diagonal $\Delta=\{(x,x)\}_{x\in X}$ belongs to $\mathcal{C}$.
\item $E_1\in\mathcal{C}$ implies $E_2\in\mathcal{C}$ for every $E_2\subset E_1$.
\item $E\in\mathcal{C}$ implies $E^{-1}\in\mathcal{C}$, where $E^{-1}=\{(y,x)\}_{(x,y)\in E}$.
\item $E_1,E_2\in\mathcal{C}$ implies $E_1\cup E_2\in\mathcal{C}$.
\item $E, F\in\mathcal{C}$ implies $E\circ F\in\mathcal{C}$, where $E\circ F$
consists of $(x,y)$ such that there is $z\in X$ so that $(x,z)\in E$
and $(z,y)\in F$.
\end{enumerate}

Recall that the {\bf star} $st(B,\mathcal{U} )$ of a subset $B$ of $X$ with respect
to a family $\mathcal{U}$ of subsets of $X$ is the union of those elements of $\mathcal{U}$
that intersect $B$.
Given two families $\mathcal{B}$ and $\mathcal{U}$ of subsets of $X$,
$st(\mathcal{B},\mathcal{U})$ is the family $\{st(B,\mathcal{U})\}$, $B\in\mathcal{B}$,
of all stars of elements of $\mathcal{B}$ with respect to $\mathcal{U}$.

\begin{Definition}\label{LSStructureDef} \cite{DH} 
A {\bf large scale structure} $\mathcal{LSS}_X$ on a set $X$ is a non-empty set of families $\mathcal{B}$
of subsets of $X$ (called {\bf uniformly bounded} 
or {\it uniformly $\mathcal{LSS}_X$-bounded} once $\mathcal{LSS}_X$ is fixed)
satisfying the following conditions:
\begin{enumerate}
\item $\mathcal{B}_1\in\mathcal{LSS}_X$ implies $\mathcal{B}_2\in\mathcal{LSS}_X$ if each element of $\mathcal{B}_2$
consisting of more than one point
is contained in some element of $\mathcal{B}_1$.
\item $\mathcal{B}_1,\mathcal{B}_2\in\mathcal{LSS}_X$ implies $st(\mathcal{B}_1,\mathcal{B}_2)\in\mathcal{LSS}_X$.
\end{enumerate}
\end{Definition}

As described in \cite{DH} the transition between the two structures is as follows:\\
1. Given a uniformly bounded cover $\mathcal{U}$ of $X$, the set $\bigcup\limits_{B\in \mathcal{U}}B\times B$ is a controlled set,\\
2. Given a controlled set $E$, the family $\{E[x]\}_{x\in X}$ is uniformly bounded,
where $E[x]:=\{y\in X\mid (x,y)\in E\}$.

\begin{Definition}\label{MediumEndDef}
Suppose $(X,\mathcal{B})$ is a set $X$ equipped with a bounded structure $\mathcal{B}$.
A \textbf{medium end} in $(X,\mathcal{B})$ is a sequence $\{B_n\}_{n=1}^\infty$ of bounded non-empty subsets of $X$ with the property that for any bounded set $A$ the set $\{n\in \mathbb{N} \mid B_n\cap A\ne \emptyset\}$ is finite.
\end{Definition}

A simple coarse structure allows for defining of uniform boundedness of medium ends.

\begin{Definition}\label{MediumEndIsUniformlyBoundedDef}
Suppose $(X,\mathcal{B})$ is a set $X$ equipped with a bounded structure $\mathcal{B}$
and $\mathcal{SCS}$ is a simple course structure on $(X,\mathcal{B})$.
A medium end $\{B_n\}_{n=1}^\infty$ in $(X,\mathcal{B})$ is \textbf{uniformly bounded} if
for any choice $x_n, y_n\in B_n$ the two simple ends $\{x_n\}_{n=1}^\infty$ and $\{y_n\}_{n=1}^\infty$ are equivalent.
\end{Definition}

\begin{Definition}\label{ReflexiveSCS}
A simple coarse structure $\mathcal{SCS}$ is \textbf{reflexive} if 
for any two equivalent simple ends  $\{x_n\}_{n=1}^\infty$ and $\{y_n\}_{n=1}^\infty$ in $X$,  $\{\{x_n,y_n\}\}_{n=1}^\infty$ is a uniformly bounded medium end.
\end{Definition}

\begin{Corollary}
The simple coarse structures described in \ref{MetricSCS}, \ref{C0SCS}, \ref{MetricExtensionSCS}, \ref{SCSInducedBySubalgebras}, \ref{GroupSCS}, and \ref{HyperbolicSCS} are all reflexive.
\end{Corollary}

\begin{Proposition}\label{UnionOfUniformlyBoundedMediumEnds}
Suppose $(X,\mathcal{B})$ is a set $X$ equipped with a bounded structure $\mathcal{B}$
and $\mathcal{SCS}$ is a reflexive simple course structure on $(X,\mathcal{B})$. Given two uniformly bounded medium ends $\{B_n\}_{n=1}^\infty$ and $\{C_n\}_{n=1}^\infty$ in $X$, the union $\{B_n\cup C_n\}_{n=1}^\infty$ is a uniformly bounded medium end if $B_n\cap C_n\ne\emptyset$ for each $n\ge 1$.
\end{Proposition}
\begin{proof}
Since $B_n\cap C_n\ne\emptyset$ for each $n\ge 1$, $\{B_n\cup C_n\}_{n=1}^\infty$ is a medium end.
Suppose $t_n\in B_n\cup C_n$ for each $n$. Pick $z_n\in B_n\cap C_n$ for each $n\ge 1$
so that $z_n=t_n$ if $t_n\in B_n\cap C_n$.
Given $n\ge 1$ we define $x_n$ to be
equal to $t_n$ if $t_n\in B_n$, otherwise $x_n=z_n$. Similarly, we define $y_n$ to be
equal to $t_n$ if $t_n\in C_n$, otherwise $y_n=z_n$. Notice $\{x_n\}_{n=1}^\infty$, $\{y_n\}_{n=1}^\infty$, and $\{z_n\}_{n=1}^\infty$ are three equivalent 
simple ends, hence $\{\{x_n,y_n\}\}_{n=1}^\infty$ is a uniformly bounded medium end.
Since $t_n\in \{\{x_n,y_n\}\}$ for each $n$, $\{t_n\}_{n\ge 1}$ is equivalent to $\{z_n\}_{n\ge 1}$.

What that shows is that any simple end $t_n\in B_n\cup C_n$ is equivalent
to one that is contained in $B_n\cap C_n$, hence two simple ends $\{t_n\}_{n\ge 1}$, $\{s_n\}_{n\ge 1}$ with $s_n, t_n\in B_n\cup C_n$
for each $n\ge 1$ are equivalent and $\{B_n\cup C_n\}_{n=1}^\infty$ is a uniformly bounded medium end.
\end{proof}

\begin{Definition}
Suppose $\mathcal{LSS}$ is a large scale structure on a set $X$. The \textbf{induced
simple coarse structure} $\mathbb{SCS}(\mathcal{LSS})$ is defined as follows:\\
1. bounded sets in $\mathbb{SCS}(\mathcal{LSS})$ are the same as bounded sets in $\mathcal{LSS}$,\\
2. two simple ends $\{x_n\}_{n=1}^\infty$ and $\{y_n\}_{n=1}^\infty$ are equivalent if 
and only if $\{\{x_n,y_n\}\}_{n=1}^\infty$ is uniformly bounded in $\mathcal{LSS}$.
\end{Definition}

\begin{Corollary}
Any simple coarse structure $\mathcal{SCS}$ on a set $X$ induced by a large scale structure $\mathcal{LSS}$ is reflexive.
\end{Corollary}

\begin{Proposition}
Suppose $\mathcal{SCS}$ is a reflexive simple coarse structure on a set $X$. The familiy of all covers $\mathcal{U}$ of $X$ satisfying conditions 1) and 2) below forms a large scale structure on $X$
whose bounded sets are the same as $\mathcal{SCS}$:\\
1. $st(B,\mathcal{U})$ is bounded for each bounded set $B$,\\
2. Any medium end consisting of elements of $\mathcal{U}$ is uniformly bounded in $\mathcal{SCS}$.
\end{Proposition}
\begin{proof}
It suffices to show that, given two covers $\mathcal{U}$ and $\mathcal{V}$ of $X$ satisfying conditions 1) and 2), the cover $st(\mathcal{U},\mathcal{V}):=\{st(U,\mathcal{V}) \mid U\in \mathcal{U}\}$ satisfies conditions 1) and 2).

Since $st(B,st(\mathcal{U},\mathcal{V}))\subset st(st(st(B,\mathcal{V}),\mathcal{U}),\mathcal{V})$, condition 1) holds.

Suppose $\{st(U_n,\mathcal{V})\}_{n\ge 1}$, $U_n\in \mathcal{U}$ is a medium end.
Given $x_n, y_n\in st(U_n,\mathcal{V})$ pick $V_n\in \mathcal{V}$ containing $x_n$ and $W_n\in \mathcal{V}$ containing $y_n$. Applying \ref{UnionOfUniformlyBoundedMediumEnds} twice we conclude $\{V_n\cup U_n\cup W_n\}_{n\ge 1}$ is a uniformly bounded medium end.
Consequently, $\{x_n\}_{n=1}^\infty$ and $\{y_n\}_{n=1}^\infty$ are equivalent simple ends
and $\{st(U_n,\mathcal{V})\}_{n\ge 1}$ is a uniformly bounded medium end.
\end{proof}

\begin{Definition}
Suppose $\mathcal{SCS}$ is a reflexive simple coarse structure on a set $X$. The \textbf{induced
large scale structure} $\mathbb{LSS}(\mathcal{SCS})$ consists of covers $\mathcal{U}$ of $X$ satisfying the following conditions:\\
1. $st(B,\mathcal{U})$ is bounded for each bounded set $B$,\\
2. Any medium end consisting of elements of $\mathcal{U}$ is uniformly bounded.
\end{Definition}

\begin{Proposition}
Suppose $(X,d)$ is a metric space.
The coarse structure induced by the simple coarse structure $\mathcal{SCS}(C_0)_d$ (see \ref{C0SCS}) is the $C_0$ coarse structure of N.Wright \cite{Wright}.
\end{Proposition}
\begin{proof}
In the language of large scale structures, a family $\mathcal{U}$ of $d$-bounded subsets of $X$ is uniformly bounded in $C_0$ if and only if, for each $\epsilon > 0$, there is a $d$-bounded subset $B$ with the property that any $U\in \mathcal{U}$ intersecting
$X\setminus B$, $U\setminus B$ is of diameter at most $\epsilon$. Therefore, given two simple ends
$\{x_n\}_{n=1}^\infty$ and $\{y_n\}_{n=1}^\infty$ such that for some elements
$U_n$ of  $\mathcal{U}$ one has $x_n, y_n\in U_n$, one can easily see $\lim\limits_{n\to\infty} d(x_n,y_n)=0$.

Conversely, suppose $\mathcal{U}$ belongs to the large scale structure
induced by $\mathcal{SCS}(C_0)_d$ but there is $\epsilon > 0$ such that for 
each $n$-ball $B(x_0,n)$ there is $U_n\in \mathcal{U}$ with the diameter of $U_n\setminus B(x_0,n)$ being larger than $\epsilon$. Therefore, one can construct two simple ends
$\{x_n\}_{n=1}^\infty$ and $\{y_n\}_{n=1}^\infty$ such that $x_n, y_n\in U_n$ yet $\lim\limits_{n\to\infty} d(x_n,y_n)\ne 0$, a contradiction. 
\end{proof}

\begin{Definition}\label{ReflexiveLSS}
A large scale structure $\mathcal{LSS}$ is \textbf{reflexive} if $\mathbb{LSS}(\mathbb{SCS}(\mathcal{LSS}))=\mathcal{LSS}$.
\end{Definition}

\begin{Corollary}
Any large scale structure $\mathcal{LSS}$ on a set $X$ induced by a reflexive simple coarse structure $\mathcal{SCS}$ is reflexive.
\end{Corollary}

\begin{Theorem}
Any large scale structure $\mathcal{LSS}$ on a countable set $X$ whose bounded sets are exactly finite subsets of $X$ is reflexive.
\end{Theorem}
\begin{proof}
Let $\mathcal{SCS}$ be the simple coarse structure induced by $\mathcal{LSS}$
and suppose $\mathcal{U}$ is a cover of $X$ consisting of finite subsets of $X$
such that any medium end $\{U_n\}_{n\ge 1}$, $U_n\in \mathcal{U}$ for each $n\ge 1$,
is uniformly bounded in $\mathcal{SCS}$.
Choose $a_k\in U_k$ for each $k\ge 1$ and let $\{x_n\}_{n\ge 1}$ be the sequence of points
obtained as follows: $a_1$ is repeated $|U_1|$-times, $a_2$ is repeated $|U_2|$-times, and so on.
Let $\{y_n\}_{n\ge 1}$ be the sequence of points
obtained as follows: we list all points of $U_1$, then we list all points of $U_2$, and so on.
In the same way we create a medium end $\{V_n\}_{n\ge 1}$ in $\mathcal{U}$: $U_1$ is repeated $|U_1|$-times, $U_2$ is repeated $|U_2|$-times, and so on.
Since $\{V_n\}_{n\ge 1}$ is a uniformly bounded medium end
and $x_n, y_n\in V_n$ for each $n\ge 1$, $\{x_n\}_{n\ge 1}$ and $\{y_n\}_{n\ge 1}$ are equivalent
in $\mathcal{SCS}$, hence $\mathcal{W}:=\{\{x_n, y_n\}\}_{n\ge 1}$  is uniformly bounded in $\mathcal{LSS}$.
Therefore $\{st(x_n, \mathcal{W})\}_{n\ge 1}$ is uniformly bounded in $\mathcal{LSS}$.
Since $U_n\subset st(x_n, \mathcal{W})$ for each $n\ge 1$,
$\{U_n\}_{n\ge 1}$ is uniformly bounded in $\mathcal{LSS}$.

Given a uniformly bounded cover $\mathcal{V}$ in $\mathbb{LSS}(\mathcal{SCS})$ we remove multiple copies of the same bounded set and we enumerate remaining elements as $\{V_n\}_{n\ge 1}$. Notice each point $x$ of $X$ belongs to finitely many elements of $\{V_n\}_{n\ge 1}$ as otherwise we would detect multiple copies of the same set in $\{V_n\}_{n\ge 1}$.
Therefore $\{V_n\}_{n\ge 1}$ is a medium end. Since we just proved that $\{V_n\}_{n\ge 1}$ is uniformly bounded in $\mathcal{LSS}$, $\mathcal{V}$ is uniformly bounded in $\mathcal{LSS}$.

\end{proof}

\begin{Example}
Consider an uncountable group $G$ and let $\mathcal{LSS}$ be the large scale structure on $G$ consisting of covers that refine covers $\mathcal{U}$ of the following form: there is a finite subset $F$ of $G$ containing $1_G$ and a countable subset $A$ of $G$ so that $\mathcal{U}=\{x\cdot F\}_{x\in A}\cup \{\{g\}\}_{g\in G\setminus A}$. $\mathcal{LSS}$ is not reflexive.
\end{Example}
\begin{proof}
$\mathbb{SCS}(\mathcal{LSS})$ has the same equivalency of ends as the standard left-invariant structure $G_l$ on $G$ (see \ref{GroupSCS}). However, $\mathbb{LSS}(G_l)$
is a bigger structure than $\mathcal{LSS}$. Namely, it is generated by covers of $G$
of the form $\{x\cdot F\}_{x\in G}$, where $F$ is a finite subset of $G$.
\end{proof}

\begin{Question}
Suppose $\mathcal{LSS}$ is a large scale structure on a set $X$ whose bounded structure has a countable basis. Is $\mathcal{LSS}$ reflexive?
\end{Question}

\section{Slowly oscillating functions and Higson corona}

In this section we generalize the concept of slowly oscillating functions from coarse structures to simple coarse structures.

\begin{Definition}
Suppose $\mathcal{LSS}$ is a large scale structure on a set $X$ so that all finite subset of $X$ are bounded.
A function $f:X\to Y$ from $X$ to a metric space $Y$ is \textbf{slowly oscillating} if for any uniformly bounded family $\mathcal{U}$ and any $\epsilon > 0$ there is a bounded
subset $B$ of $X$ such that for any $U\in \mathcal{U}$ missing $B$ one has
$\diam(f(U)) < \epsilon$.
\end{Definition}

\begin{Definition}
Suppose $\mathcal{SCS}$ is a simple coarse structure on a set $X$.
A function $f:X\to Y$ from $X$ to a metric space $Y$ is \textbf{slowly oscillating} if for any
two simple ends $\{x_n\}_{n=1}^\infty$ and $\{y_n\}_{n=1}^\infty$ that are equivalent
one has $d_Y(f(x_n),f(y_n))\to 0$ as $n\to \infty$.
\end{Definition}

\begin{Proposition}
Suppose $\mathcal{LSS}$ is a large scale structure on a set $X$ so that all finite subset of $X$ are bounded.
If a function $f:X\to Y$ from $X$ to a metric space $Y$ is slowly oscillating, then it is slowly oscillating from the point of view of the induced simple scale structure on $X$.
\end{Proposition}

\begin{proof}
Suppose $f$ is slowly oscillating with respect to $\mathcal{LSS}$ and $\{x_n\}_{n=1}^\infty$ and $\{y_n\}_{n=1}^\infty$ are two simple ends that are equivalent. Therefore, $\mathcal{U}:=\{\{x_n,y_n\}\}_{n\ge 1}$ is a uniformly bounded family and for any $\epsilon > 0$ there is a bounded set $B$ in $X$ such that if both $x_n$ and $y_n$ are outside $B$, then $d_Y(f(x_n),f(y_n)) < \epsilon$. Since that holds for all but finitely many $n$, $d_Y(f(x_n),f(y_n))\to 0$ as $n\to \infty$.
\end{proof}

\begin{Proposition}
Suppose $\mathcal{SCS}$ is a reflexive simple coarse structure on a set $X$ so that every finite subset of $X$ is bounded and there is a countable basis of bounded sets in $X$.
If a function $f:X\to Y$ from $X$ to a metric space $Y$ is slowly oscillating, then it is slowly oscillating from the point of view of the induced large scale structure on $X$.
\end{Proposition}

\begin{proof}
Suppose $f$ is slowly oscillating but is not slowly oscillating from the point of view of the induced large scale structure on $X$. That means there is a uniformly bounded family
$\mathcal{U}$ and $\epsilon > 0$ such that for any bounded subset $B$ of $X$
there is $U_n\in\mathcal{U}$ missing $B$ so that $\diam(f(U_n)) > \epsilon$.
Choose an increasing sequence $B_n$ of bounded subsets of $X$ that is a basis for all bounded subsets of $X$. Choose $U_n\in\mathcal{U}$ missing $B_n$ so that $\diam(f(U_n)) > \epsilon$. Pick points $x_n, y_n\in U_n$ satisfying
$d_Y(f(x_n),f(y_n)) \ge \epsilon$. Notice $\{x_n\}_{n=1}^\infty$ and $\{y_n\}_{n=1}^\infty$  are equivalent simple ends
and $d_Y(f(x_n),f(y_n))\to 0$ as $n\to \infty$ fails, a contradiction.
\end{proof}

\begin{Definition}
Suppose $X$ is a locally compact Hausdorff space and $\mathcal{SCS}$ is a simple coarse structure on $X$. $\mathcal{SCS}$ is \textbf{compatible} with the topology on $X$
if the bounded sets of $\mathcal{SCS}$ coincide with pre-compact sets of $X$.
\end{Definition}

\begin{Definition}
Suppose $X$ is a locally compact Hausdorff space and $\mathcal{SCS}$ is a simple coarse structure on $X$ that is compatible with the topology on $X$. The \textbf{Higson compactification} $\bar X$ of $X$ is the compactification such that $f:\bar X\to [0,1]$ is continuous if and only $f|X$ is continuous and slowly oscillating.

The \textbf{geometric Higson corona} of $\mathcal{SCS}$ is defined as $\bar X\setminus X$.
\end{Definition}

\begin{Theorem}
Suppose $X$ is a $\sigma$-compact locally compact Hausdorff space and $\bar X$
is a compactification of $X$ such that each point of the corona $\bar X\setminus X$ has a countable basis of neighborhoods. The Higson compactification of $X$ with respect to the simple coarse structure induced by $\bar X$ is identical with $\bar X$.
\end{Theorem}
\begin{proof}
It suffices to show that slowly oscillating continuous functions on $X$ are exactly those that extend continuously over $\bar X$. \\
Given a continuous function $f:\bar X\to [0,1]$ and given two equivalent simple ends
$\{x_n\}_{n=1}^\infty$ and $\{y_n\}_{n=1}^\infty$ such that $|f(x_n)-f(y_n)|$ does not converge to $0$ as $n\to \infty$, we can (by switching to a subsequence) assume that
$\lim\limits_{n\to\infty}f(x_n)=u$, $\lim\limits_{n\to\infty}f(y_n)=v$ and $u\ne v$.
That implies that all the points in the corona of $\{x_n\}_{n=1}^\infty$ are mapped by $f$ to $u$ and all the points in the corona of $\{y_n\}_{n=1}^\infty$ are mapped by $f$ to $v$, a contradiction since those coronas are identical.

Suppose $g:X\to [0,1]$ is continuous and slowly continuous with respect to the simple coarse structure induced by $\bar X$. Given $x\in\bar X\setminus X$, consider the 
intersection of closures of all sets $f(A\cap X)$, $A$ ranging over all closed neighborhoods of $x$ in $\bar X$. That intersection must consist of exactly one point and that point is defined as the value of $g(x)$. That results in an extension of $g$ to a continuous function on $\bar X$.

Indeed, we may consider only a basis $\{A_n\}_{n\ge 1}$ of closed neighborhoods of $x$ in $\bar X$. In that case any two sequences $x_n, y_n\in A_n\setminus X$ form equivalent simple ends, so $\lim\limits_{n\to\infty}|f(x_n)-f(y_n)|=0$ and $\bigcap\limits_{n=1}^\infty cl(f(A_n\cap X))$ must be a single point.

If $\bar g$ is not continuous at $x\in \bar X\setminus X$, then there is a sequence
$z_n\in \bar X\setminus X$ converging to $x$ such that $\lim\limits_{n\to\infty}\bar g(z_n)$ exists and is not equal to $\bar g(x)$. In that case we may shadow $\{z_n\}_{n\ge 1}$ by a sequence $\{x_n\}_{n\ge 1}$ in $X$ that converges to $x$
and $\lim\limits_{n\to\infty}\bar g(z_n)= \lim\limits_{n\to\infty}g(x_n)$, a contradiction since the latter limit is $\bar g(x)$.
\end{proof}

\begin{Theorem}\label{GromovCoronaIsHigson}
If $X$ is a Gromov hyperbolic space, then
the Higson corona of the induced simple coarse structure equals the Gromov boundary of $X$ if $X$ is proper and geodesic.
\end{Theorem}
\begin{proof}
Consider the Gromov compactification $\bar X=X\cup \partial X$ of $X$ obtained by adding the Gromov boundary $\partial X$ to $X$. We need to show that the induced simple coarse structure on $X$ from $\bar X$ is identical with the one determined by the Gromov hyperbolic metric on $X$.

Let $a$ be a fixed basepoint of $X$.
As described in \cite{KapBen} (Definition 2.9), $ \partial X$ consists of equivalence classes
$[\{x_n\}_{n\ge 1}]$ of sequences $\{x_n\}_{n\ge 1}$ converging to infinity in $X$.
That means $\lim\inf\limits_{i,j\to\infty}\left<x_i, x_j\right>_a=\infty$ and two sequences converging to infinity $\{x_n\}_{n\ge 1}$ and $\{y_n\}_{n\ge 1}$ are equivalent if
$$\lim\inf\limits_{i,j\to\infty}\left<x_i, y_j\right>_a=\infty.$$
Notice that a simple end $\{x_n\}_{n\ge 1}$ converges to $p\in \partial X$ if and only if $\{x_n\}_{n\ge 1}$ converges to infinity and is equivalent to $p$.

We need to show that, given two simple ends $\{x_n\}_{n\ge 1}$ and $\{y_n\}_{n\ge 1}$
in $X$, the condition $\lim\limits_{n\to\infty}\left<x_n, y_n\right>_a =\infty$
if and only if, for each strictly increasing function $b:\mathbb{N}\to\mathbb{N}$
the coronas of closures in $X\cup\partial X$ of $\{x_{b(n)}\}_{n=1}^\infty$ and $\{y_{b(n)}\}_{n=1}^\infty$ are equal.

Suppose there is a strictly increasing function $b:\mathbb{N}\to\mathbb{N}$ and $r > 0$
such that $ \left<x_{b(n)}, y_{b(n)}\right>_a  < r$ for all $n\ge 1$. We may assume each sequence $\{x_{b(n)}\}_{n\ge 1}$ and $\{y_{b(n)}\}_{n\ge 1}$ converges to infinity. In that case they converge to different points in $\partial X$, a contradiction.

Suppose $p\in\partial X$ belongs to the closure of $\{x_{b(n)}\}_{n=1}^\infty$
for some strictly increasing function $b:\mathbb{N}\to\mathbb{N}$.
We may assume $\{x_{b(n)}\}_{n\ge 1}$ and $\{y_{b(n)}\}_{n\ge 1}$ converges to infinity,
hence it converges to $p$. Since $\lim\limits_{n\to\infty}\left<x_{b(n)}, y_{b(n)}\right>_a =\infty$ and we can choose a subsequence of $\{y_{b(n)}\}_{n\ge 1}$ converging to infinity,
$p$ also belongs to the closure of $\{y_{b(n)}\}_{n\ge 1}$.
\end{proof}

\section{Freundenthal compactification is a Higson compactification}
In this section we describe how Freundenthal compactification can be viewed from the point of view of simple coarse theory.
\begin{Definition}
Suppose $X$ is a $\sigma$-compact locally compact and locally connected Hausdorff space. Two simple ends $\{x_n\}_{n=1}^\infty$ and $\{y_n\}_{n=1}^\infty$ are \textbf{Freundenthal-equivalent} if and only if there is no compact subset $K$ of $X$ and a component $U$ of $X\setminus K$ that contains infinitely many points of one end and only finitely many points of the other end.

Equivalently, for any strictly increasing function $a:\mathbb{N}\to\mathbb{N}$ and for any compact subset $K$ of $X$
the following conditions hold for any component $U$ of $X\setminus K$:\\
1. $\{x_{a(n)}\}_{n=1}^\infty\subset U$ implies $y_{a(n)}\in U$ for almost all $n$,\\
2. $\{y_{a(n)}\}_{n=1}^\infty\subset U$ implies $x_{a(n)}\in U$ for almost all $n$.
\end{Definition}

\begin{Theorem}
Suppose $X$ is a $\sigma$-compact locally compact and locally connected Hausdorff space. The Higson compactification $\bar X$ of the simple coarse structure on $X$ induced by the
Freundenthal-equivalence is the Freundenthal compactification of $X$.
\end{Theorem}
\begin{proof}
What we need to prove (see \cite{Peschke}) is that
$\bar X\setminus X$ is of dimension $0$ and $\bar X$ dominates any compactification
$\hat X$ of $X$ whose corona is of dimension $0$.

Suppose $\hat X$ is a compactification of $X$ whose corona is of dimension $0$.
Consider a continuous $f:\hat X\to [0,1]$ and assume there are two simple ends $\{x_n\}_{n=1}^\infty$ and $\{y_n\}_{n=1}^\infty$ that are Freundenthal-equivalent but,
for some $\epsilon > 0$, one has $|f(x_n)-f(y_n)| > \epsilon$ for all $n\ge 1$.
Therefore coronas of $\{x_n\}_{n=1}^\infty$ and $\{y_n\}_{n=1}^\infty$ are disjoint and 
there are two disjoint open sets $U$ and $V$ of $\hat X$ such that
$\hat X\setminus X\subset U\cup V$, the corona of $\{x_n\}_{n=1}^\infty$ is contained in $U$, and the corona of $\{y_n\}_{n=1}^\infty$ is contained in $V$.
Put $K=\hat X\setminus (U\cup V)$ and notice the existence of a component of $X\setminus K$ that contains an infinite subsequence $\{x_{a(n)}\}_{n=1}^\infty$ 
but none of $\{y_{a(n)}\}_{n=1}^\infty$, a contradiction.

Pick an increasing sequence $\{K_n\}_{n\ge 1}$ of compact subsets of $X$ whose union is $X$ and $K_n\subset int(K_{n+1}$ for all $n\ge 1$.
Notice that the corona of each component $U$ of $X\setminus K_n$, $n\ge 1$, is open-closed in $\bar X\setminus X$. Indeed there is a continuous slowly oscillating function on $X$ that equals $0$ on $U\setminus K_{n+1}$ and equals $1$ on $X\setminus (U\cup K_{n+1}$. Its extension over $\bar X$ has only two values on $\bar X\setminus X$: $0$ on the corona of $U$ and $1$ on its complement in $\bar X\setminus X$.

Given any decreasing sequence $\{U_n\}_{n\ge 1}$ of components of $X\setminus K_n$,
the intersection of their closures in $\bar X$ must be a point and each point in $\bar X\setminus X$ belongs to such intersection. Thus $\bar X\setminus X$ is of dimension $0$.
Indeed, given any slowly oscillating continuous function $f:X\to [0,1]$, the diameters of $f(U_n)$ must converge to $0$ as otherwise there are Freundenthal-equivalent ends
$\{x_n\}_{n=1}^\infty$ and $\{y_n\}_{n=1}^\infty$ such that $|f(x_n)-f(y_n)|$ does not converge to $0$.
\end{proof}

\section{Applications}

In this section we extend results of Mine-Yamashita \cite{MineYama} while providing very simple proofs of them.

First, we will translate concepts from coarse theory to simple coarse theory.

\begin{Proposition}\label{BornologousInSCC}
Suppose $\mathcal{SCS}_X$ is a reflexive simple coarse structure on a set $X$ inducing the
large scale structure $\mathbb{LSS}(\mathcal{SCS}_X)$ so that every unbounded subset of $X$ contains a simple end. Suppose $\mathcal{SCS}_Y$ is a reflexive  simple coarse structure on a set $Y$ inducing the
large scale structure $\mathbb{LSS}(\mathcal{SCS}_Y)$. If $f$ is a function from $X$ to $Y$, then the following conditions are equivalent:\\
1. $f$ is bornologous when considered as a function from $(X,\mathbb{LSS}(\mathcal{SCS}_X))$ to $(Y,\mathbb{LSS}(\mathcal{SCS}_Y))$,\\
2. $f$ preserves bounded sets and, for any two equivalent simple ends $\{x_n\}_{n=1}^\infty$ and $\{y_n\}_{n=1}^\infty$ in $X$, either one of the sequences $\{f(x_n)\}_{n=1}^\infty$, $\{f(y_n)\}_{n=1}^\infty$ is not a simple end in $Y$ or both of them are simple ends in $Y$ and they are equivalent.
\end{Proposition}
\begin{proof}
$f$ being bornologous means it preserves uniformly bounded families.

1)$\implies$2). 
Given any two simple ends
$\{x_n\}_{n=1}^\infty$ and $\{y_n\}_{n=1}^\infty$ in $X$, $\{\{x_n,y_n\}\}_{n\ge 1}$ is uniformly bounded, hence $\{\{f(x_n),f(y_n)\}\}_{n\ge 1}$ is uniformly bounded
and $\{f(x_n)\}_{n=1}^\infty$, $\{f(y_n)\}_{n=1}^\infty$ are equivalent in $Y$ if they are simple ends.

2)$\implies$1). 
Suppose $\mathcal{U}$ is a uniformly bounded cover of $X$. To show $f(\mathcal{U})$ is uniformly bounded consider a medium end $\{f(U_n)\}_{n\ge 1}$ in $Y$, where $U_n\in\mathcal{U}$.
Notice $\{U_n\}_{n\ge 1}$ is a medium end in $X$.
Given $x_n, y_n\in U_n$, both $\{x_n\}_{n=1}^\infty$ and $\{y_n\}_{n=1}^\infty$ are equivalent simple ends in $X$. Therefore $\{f(x_n)\}_{n=1}^\infty$ and $\{f(y_n)\}_{n=1}^\infty$
are equivalent simple ends in $Y$ proving that $\{f(U_n)\}_{n\ge 1}$ is uniformly bounded in $Y$.
\end{proof}

\begin{Corollary}
Suppose $\mathcal{SCS}_X$ is a reflexive simple coarse structure on a set $X$ inducing the
large scale structure $\mathbb{LSS}(\mathcal{SCS}_X)$ so that every unbounded subset of $X$ contains a simple end. Suppose $\mathcal{SCS}_Y$ is a reflexive  simple coarse structure on a set $Y$ inducing the
large scale structure $\mathbb{LSS}(\mathcal{SCS}_Y)$. If $f$ is a function from $X$ to $Y$, then the following conditions are equivalent:\\
1. $f$ is coarse bornologous when considered as a function from $(X,\mathbb{LSS}(\mathcal{SCS}_X))$ to $(Y,\mathbb{LSS}(\mathcal{SCS}_Y))$,\\
2. $f$ preserves bounded sets, preserves simple ends, and preserves equivalence of ends.
\end{Corollary}
\begin{proof}
$f$ being coarse bornologous means it co-preserves bounded sets and preserves uniformly bounded families.

1)$\implies$2). Suppose $\{x_n\}_{n=1}^\infty$ is a simple end in $X$ but $\{f(x_n)\}_{n=1}^\infty$
is not a simple end in $Y$. Therefore, there is a bounded subset $B$ of $Y$ such that
$\{n\in \mathbb{N}\mid f(x_n)\in B\}$ is infinite. Consequently, $\{n\in \mathbb{N}\mid x_n\in f^{-1}(B)\}$ is infinite, a contradiction as $f^{-1}(B)$ is a bounded subset of $X$.

2)$\implies$1). $f$  co-preserves bounded sets. Indeed, if $f^{-1}(B)$ is unbounded in $X$ for some bounded subset $B$ of $Y$, then it contains a simple end that cannot be mapped to a simple end in $Y$.
\end{proof}

\begin{Corollary}
Suppose $\mathcal{SCS}_X$ is a reflexive simple coarse structure on a set $X$ inducing the
large scale structure $\mathbb{LSS}(\mathcal{SCS}_X)$ so that every unbounded subset of $X$ contains a simple end. Suppose $\mathcal{SCS}_Y$ is a reflexive  simple coarse structure on a set $Y$ inducing the
large scale structure $\mathbb{LSS}(\mathcal{SCS}_Y)$. If $f, g$ are functions from $X$ to $Y$
preserving bounded sets, preserving simple ends, and preserving equivalence of ends, then the following conditions are equivalent:\\
1. $f$ and $g$ are close coarse bornologous functions when considered as functions from $(X,\mathbb{LSS}(\mathcal{SCS}_X))$ to $(Y,\mathbb{LSS}(\mathcal{SCS}_Y))$,\\
2. For each simple end $\{x_n\}_{n=1}^\infty$ in $X$,
$\{f(x_n)\}_{n=1}^\infty$ and $\{g(x_n)\}_{n=1}^\infty$ are simple ends in $Y$ and they are equivalent.
\end{Corollary}
\begin{proof}
$f$ and $g$ being close means that $\{f(x),g(x)\}_{x\in X}$ is a uniformly bounded family in $Y$.
Therefore 1$\implies$2) is obvious.

2)$\implies$1). Given a medium end $\{f(x_n),g(x_n)\}_{n\ge 1}$, $\{x_n\}_{n=1}^\infty$ is a simple end in $X$. Since $\{f(x_n)\}_{n=1}^\infty$ and $\{g(x_n)\}_{n=1}^\infty$ are equivalent and the simple coarse structure on $Y$ is reflexive, $\{f(x_n),g(x_n)\}_{n\ge 1}$ is uniformly bounded.
Thus, $\{f(x),g(x)\}_{x\in X}$ is a uniformly bounded family in $Y$.
\end{proof}

\begin{Example}
Given a Frechet space $X$ (see \cite{En}), let $\mathcal{B}$ be the family of all its finite subsets.
Two simple ends in $(X,\mathcal{B})$ are declared equivalent if either they are identical or they converge to the same point. Then, any function $f:X\to Y$ between Frechet spaces is bornologous if it is continuous.
\end{Example}

\begin{Definition}
A function $f:X\to Y$ of metric spaces is \textbf{uniformly continuous} on a subset $A$ of $X$
if for every $\epsilon > 0$ there is $\delta > 0$ such that $d_X(a,b) < \delta$ and $a\in A$
implies $d_Y(f(a),f(b)) < \epsilon$.
\end{Definition}

\begin{Theorem}
Suppose $(\bar X,d_X)$, $(\bar Y,d_Y)$ are metric spaces and $X\subset \bar X$, $Y\subset \bar Y$ are proper dense subsets. Consider induced simple coarse structures on $X$ and $Y$ as described in \ref{MetricExtensionSCS}.\\
A function $f:X\to Y$ preserving bounded sets is coarse bornologous if
$f$ extends to $\bar f:\bar X\to \bar Y$ that is uniformly continuous on $\bar X\setminus X$ and $\bar f(\bar X\setminus X)\subset \bar Y\setminus Y$.
\end{Theorem}
\begin{proof}
Let $A:=\bar X\setminus X$ and $B:=\bar Y\setminus Y$.
Given a sequence $\{x_n\}_{n\ge 1}$ in $X$ converging to a point $a\in A$, $\bar f(x_n)$ converges to $\bar f(a)$. Therefore point-inverse under $\bar f$ of a bounded set in $Y$ is bounded in $X$.

Suppose
$\{x_n\}_{n=1}^\infty$ and $\{y_n\}_{n=1}^\infty$ are two equivalent simple ends in $X$
such that $\{\bar f(x_n)\}_{n=1}^\infty$ and $\{\bar f(y_n)\}_{n=1}^\infty$ are not equivalent.
Without loss of generality we may assume existence of $\epsilon > 0$ such that
$d_Y(f(x_n),f(y_n)) > 2\cdot \epsilon$ for each $n\ge 1$. Find $\delta > 0$ such that
$d_Y(\bar f(a),\bar f(b)) < \epsilon$ if $a\in A$ and $b\in \bar X$ are at distance less than $\delta$.
We can find $n$ large enough that $d_X(x_n,y_n) < \delta/2$ and there is $a\in A$ so that $d_X(a,x_n) < \delta/ 2$. Therefore, $d_Y(\bar f(a),\bar f (x_n)) <\epsilon$ and $d_Y(\bar f(a),\bar f (y_n)) <\epsilon$ resulting in $d_Y(\bar f(y_n),\bar f (x_n)) < 2\cdot \epsilon$, a contradiction.
Thus, $f:X\to Y$ is coarse bornologous.
\end{proof}

\begin{Theorem}
Suppose $(\bar X,d_X)$, $(\bar Y,d_Y)$ are metric spaces and $X\subset \bar X$, $Y\subset \bar Y$ are proper dense subsets. Consider induced simple coarse structures on $X$ and $Y$ as described in \ref{MetricExtensionSCS}.\\
A coarse bornologous function $f:X\to Y$ extends uniquely to $\bar f:\bar X\to \bar Y$ that is uniformly continuous on $\bar X\setminus X$ and $\bar f(\bar X\setminus X)\subset \bar Y\setminus Y$ if $\bar Y\setminus Y$ is complete.
\end{Theorem}
\begin{proof}
Suppose $f:X\to Y$ is coarse bornologous. Given $x\in A$ choose a sequence $\{x_n\}_{n\ge 1}$ in $X$ converging to $x$. Notice $\{f(x_n)\}_{n\ge 1}$ is a Cauchy sequence.
Choose $y_n\in B$ so that $d_Y(f(x_n),y_n) < 2\cdot dist(f(x_n),B)$. Notice $\{y_n\}_{n\ge 1}$ is a Cauchy sequence in $B$, hence it converges to $y\in B$. We define $\bar f(x)$ to be $y$. 
$\bar f$ is well-defined as $f$ preserves equivalency of simple ends.

Notice $\bar f$ is continuous at each point of $A$ due to the way $\bar f$ was defined.
If $\bar f$ is not uniformly continuous on $A$, we can find a simple end $\{x_n\}_{n\ge 1}$ in $X$
and a sequence $\{y_n\}_{n\ge 1}$ in $B$ such that $\lim\limits_{n\to\infty}d_X(x_n,y_n)=0$
yet $d_Y(f(x_n),\bar f(y_n)) > \epsilon$ for all $n\ge 1$ and some fixed $\epsilon > 0$.
For each $n\ge 1$ we can find $z_n\in X$ so that $d_X(y_n,z_n) < d_X(y_n,x_n)$
and $d_Y(\bar f(y_n),f(z_n)) < \epsilon/2$. In that case
$\{z_n\}_{n\ge 1}$ is a simple end equivalent to $\{x_n\}_{n\ge 1}$. However, simple ends
$\{f(z_n)\}_{n\ge 1}$ and $\{f(x_n)\}_{n\ge 1}$ are not equivalent, a contradiction.
\end{proof}

\begin{Theorem}\label{ExtendingFunctionsFromTheBoundary}
Suppose $(\bar X,d_X)$, $(\bar Y,d_Y)$ are metric spaces and $X\subset \bar X$, $Y\subset \bar Y$ are proper dense subsets. Consider induced simple coarse structures on $X$ and $Y$ as described in \ref{MetricExtensionSCS}.\\ 
If $Y$ is open in $\bar Y$ and there is a sequence of subsets $Y_n$
converging to $\bar Y\setminus Y$ in the Hausdorff metric such that for some decreasing sequence
$\epsilon_n$ of positive numbers $Y_n$ is outside the $\epsilon_n$-ball around $\bar Y\setminus Y$, then
any uniformly continuous function $g:\bar X\setminus X\to \bar Y\setminus Y$ extends uniquely
(up to closeness) to a function $\bar g:\bar X\to \bar Y$
such that \\
1. $\bar g(X)\subset Y$, \\
2. $\bar g$ is uniformly continuous on $\bar X\setminus X$,\\
3. $\bar g|X:X\to Y$ is a coarse bornologous function.
\end{Theorem}
\begin{proof}
Let $A:=\bar X\setminus X$ and $B:=\bar Y\setminus Y$.
Send all points in $X$ whose distance to $A$ is at least $1$ to some fixed point $y_0\in Y$.
Given $m > 0$ and $x\in X$ at the distance from $A$ belonging to $[1/(m+1),1/m)$, find $n$ such that
the Hausdorff distance from $Y_n$ is smaller than $1/m$, then
find a point $x'\in A$ so that $d_X(x,A) > 0.5\cdot d_X(x,x')$, and finally
 find a point $y'\in Y_n$ such that $d_Y(g(x'),y') < 1/m$.
Define $\bar g(x)$ as $y'$. Notice $\bar g$ is uniformly continuous on $\bar X\setminus X$.

Notice $\{x_n\}_{n=1}^\infty$ is a simple end in $X$ if and only if $\lim\limits_{n\to\infty}d_X(x_n,A)=0$. From the construction we get that $\lim\limits_{n\to\infty}d_Y(\bar g(x_n),B)=0$.
That means $\bar g|X$ preserves simple ends. In particular, it co-preserves bounded sets.

From the construction we can see that $\bar g|X:X\to Y$ preserves bounded sets. Thus, $\bar g|X$ is a coarse bornologous function.

If $h:\bar X\to \bar Y$ is another extension of $g$ satisfying Conditions 1)-3) above,
then all we have to show that $\{h(x),\bar g(x)\}_{x\in X}$ is a uniformly bounded family in $Y$.
Indeed, otherwise there is a sequence $\{x_n\}_{n\ge 1}$ in $X$ converging to $a\in A$
such that $\{h(x_n)\}_{n\ge 1}$ is not equivalent to $\{\bar g(x_n)\}_{n\ge 1}$. However, both these sequences converge to $g(a)$, a contradiction.
\end{proof}

\begin{Corollary}
Suppose $(\bar X,d_X)$, $(\bar Y,d_Y)$ are metric spaces and $X\subset \bar X$, $Y\subset \bar Y$ are proper dense subsets. Consider induced simple coarse structures on $X$ and $Y$ as described in \ref{MetricExtensionSCS}.\\ 
If $Y$ is open in $\bar Y$ and and $\bar Y\setminus Y$ is totally bounded in the metric $d_Y$, then
any uniformly continuous function $g:\bar X\setminus X\to \bar Y\setminus Y$ extends uniquely
(up to closeness) to a function $\bar g:\bar X\to \bar Y$
such that \\
1. $\bar g(X)\subset Y$, \\
2. $\bar g$ is uniformly continuous on $\bar X\setminus X$,\\
3. $\bar g|X:X\to Y$ is a coarse bornologous function.
\end{Corollary}
\begin{proof}
Given $n\ge 1$ cover $\bar Y\setminus Y$ by finitely many balls of radius $1/n$ centered in some points of $\bar Y$. In each ball select a point belonging to $Y$ and declare the set of all such points to be $Y_n$. Notice $\{Y_n\}_{n\ge 1}$
converges to $\bar Y\setminus Y$ in the Hausdorff metric such that for some decreasing sequence
$\epsilon_n$ of positive numbers $Y_n$ is outside the $\epsilon_n$-ball around $\bar Y\setminus Y$.
\end{proof}

\begin{Definition}\label{UniformDimensionDef}
Let $(X,d)$ be a metric space and $n\ge 0$. The \textbf{uniform dimension} of $(X,d)$ is at most $n$ if every uniform cover $\mathcal{U}$ of $X$ has a uniform refinement $\mathcal{V}$ such that each point $x\in X$ belongs to at most $n+1$ elements of $\mathcal{V}$.

A uniform cover $\mathcal{W}$ of $X$ is a cover such that for some $r > 0$ the cover $\{B(x,r)\}_{x\in X}$
of $X$ by all $r$-balls refines $\mathcal{W}$. Thus, all open covers of a compact metric space $X$ are uniform and the uniform dimension of $X$ equals its covering dimension.
\end{Definition}

\begin{Theorem}\label{AsdimVsUniformDimThm}
Suppose $(\bar X,d_X)$ is a metric space and $X\subset \bar X$ is a proper open dense subset such that there is a sequence of subsets $X_n$
converging to $\bar X\setminus X$ in the Hausdorff metric so that for some decreasing sequence
$\epsilon_n$ of positive numbers each $X_n$ is outside the $\epsilon_n$-ball around $\bar X\setminus X$. Consider the induced simple coarse structure on $X$ as described in \ref{MetricExtensionSCS}.\\ 
1.) The uniform dimension of $\bar X\setminus X$ is at most the asymptotic dimension of $X$.\\
2). If the uniform dimension of $\bar X\setminus X$ equals $k\ge 0$, then
the asymptotic dimension of $X$ is at most $2+3\cdot k$.
\end{Theorem}
\begin{proof}
Let $A:= \bar X\setminus X$ and $\bar A:=A\times [0,1]$ with the $l_1$-metric. Using \ref{ExtendingFunctionsFromTheBoundary} we can see that $X$ is coarsely equivalent
to $\bar A\setminus A\times\{0\}$, so we may simply assume $\bar X=A\times [0,1]$
and $X=A\times (0,1]$.\\
1).
Assume the asymptotic dimension of $X$ is $k$.
Consider the cover of $X$ defined as $\{B(x,1/n)\times [1/(n+1),1/n]\}_{x\in X, n\ge 1}$.
It has a uniformly bounded coarsening $\mathcal{V}$ such that each point $x\in X$ belongs to at most $k+1$ elements of $\mathcal{V}$. Given $r > 0$, there is $t\in (0,1]$ with the property
that every element of $\mathcal{V}$ intersecting $A\times \{t\}$ is of diameter less than $1/(2r)$.
Consequently, the restriction of $\mathcal{V}$ to $A\times \{t\}$ gives a uniform refinement
of the cover of $A\times \{t\}$ by $r$-balls such that every point in $A\times \{t\}$ belongs to at most $k+1$ elements of that cover.\\
2). Suppose the uniform dimension of $A$ is $k$. Find a decreasing sequence of positive numbers
$\{\delta_n\}_{n\ge 1}$ with the property that the cover of $A$ by $1/n$-balls has a refinement
$\mathcal{V}_n$ and the cover of $A$ by $\delta_n$-balls is a refinement of $\mathcal{V}_n$.
Moreover, each point $x\in A$ belongs to at most $k+1$ elements of $\mathcal{V}_n$.

Suppose $\mathcal{U}$ is a uniformly bounded cover of $A\times (0,1]$. Put $\mathcal{V}_{-1}=\mathcal{V}_0=\{A\}$, $\delta_0=\delta_{-1}=\infty$, and $\mu_0:=1=\mu_{-1}$.
We plan to find a strictly decreasing sequence $\{\mu_n\}_{n\ge 0}$ in $(0,1]$
and a strictly increasing sequence $\{\alpha(n)\}_{n\ge 0}$ of non-negative integers
so that the cover $\mathcal{W}:=\bigcup\limits_{n=0}^\infty \mathcal{V}_{\alpha(n)}\times (\mu_{n+2},\mu_{n-1}]$ coarsens $\mathcal{U}$. Obviously, $\mathcal{W}$ is a uniformly bounded cover of $A\times (0,1]$ and each point belongs to at most $3k+3$ elements of $\mathcal{W}$.

The crucial property we want to achieve is that if $U\in \mathcal{U}$ intersects
$A\times [\mu_{n+1},\mu_n]$, then it is contained in $A\times (\mu_{n+2},\mu_{n-1})$
and $U$ is of diameter at most $\delta_{\alpha(n-1)}/2$ which will guarantee that 
$U$ is contained in an element of $\mathcal{V}_{\alpha(n-1)}\times (\mu_{n+2},\mu_{n-1}]$.

Put $\mu_1=1/2$, $\alpha(0)=0$, and choose $\mu_2 < \mu_1$ so that
$st(A\times [\mu_1,\mu_0],\mathcal{U})\subset A\times (\mu_2,1]$.

Inductive step: Given $\mu_{n+1}$ for some $n\ge 1$ find $t > 0$ so that
$st(A\times [\mu_{n+1},1],\mathcal{U})\subset A\times (t,1]$. Find $m > \alpha(n+1)$ to achieve
$\delta_m < \mu_{n+1}-\mu_{n}$ and put $\alpha(n+2)=m$.
Then find $\mu_{n+2} < t$ so that $U\cap (A\times (0,\mu_{n+2}])\ne\emptyset$, $U\in \mathcal{U}$,
implies $diam(U) < \delta_{\alpha(n+2)}/2$.

Now, if $U\in \mathcal{U}$ intersects
$A\times [\mu_{n+2},\mu_{n+1}]$, then it is contained in $A\times (\mu_{n+3},1])$ and,
since $diam(U) < \delta_{\alpha(n)}/2 < \mu_{n+1}-\mu_{n}$,
$U\subset A\times (\mu_{n+3},\mu_{n}]$.
\end{proof}

\begin{Question}
Suppose $(A,d_A)$ is a metric space and $X:=A\times (0,1]\subset \bar X:=A\times [0,1]$. Consider the induced simple coarse structure on $X$ as described in \ref{MetricExtensionSCS}. If the uniform dimension of $A$ equals $k\ge 0$, is
the asymptotic dimension of $X$ is at most $k+1$?
\end{Question}

\begin{Question}
Suppose $(A,d_A)$ is a compact metric space and $X:=A\times (0,1]\subset \bar X:=A\times [0,1]$. Consider the induced simple coarse structure on $X$ as described in \ref{MetricExtensionSCS}. If the covering dimension of $A$ equals $k\ge 0$, is
the asymptotic dimension of $X$ is at most $k+1$?
\end{Question}

\begin{Observation}
One can use the idea of \ref{AsdimVsUniformDimThm} to introduce new topological properties of compact metric spaces $X$ via coarse properties of $X\times (0,1]$ equipped with the simple coarse structure as described in \ref{MetricExtensionSCS}.
In case a coarse property $\mathcal{P}$ is defined for metric spaces $M$ only one can generalize it to arbitrary large scale spaces $Y$ as follows: $Y$ has property $\mathcal{P}$ if given a uniformly bounded cover $\mathcal{U}_1$ of $Y$ there is a sequence $\{\mathcal{U}_n\}_{n\ge 1}$
of uniformly bounded covers of $Y$ so that $st(\mathcal{U}_n,\mathcal{U}_n)$ refines $\mathcal{U}_{n+1}$ and $Y$ equipped with the coarse structure generated by $\{\mathcal{U}_n\}_{n\ge 1}$ (which is metrizable) has property $\mathcal{P}$.

\end{Observation}

\section{Uniform structures on subsets of ends}

In \cite{Har} a uniform structure is defined on the space of ends that can be extended to a coarse embedding from the natural numbers to $X$. Proposition 60 there claims this uniformity has a countable base. However the second part of the proof directly contradicts the following:
\begin{Proposition}
Suppose $S\subset \mathbb{R}_+^\mathbb{N}$. If, for every $\xi:\mathbb{N}\to \mathbb{R}_+$ there is $\lambda\in S$ and $R\ge 0$ such that
$$\xi(i)-\lambda(i) \leq R$$
for all $i$, then $S$ is uncountable.
\end{Proposition}
\begin{proof}
Assume $S$ is countable. We may assume $R$ is always natural. Also, we may replace $\mathbb{N}$ by any countable set, so let us replace it by $\mathbb{N}\times \mathbb{N}$.
Thus, $S\subset \mathbb{R}_+^{\mathbb{N}\times \mathbb{N}}$ and, for every $\xi:\mathbb{N}\times \mathbb{N}\to \mathbb{R}_+$, there is $\lambda\in S$ and $R\in \mathbb{N}$ such that
$$\xi(i)-\lambda(i) \leq R$$
for all $i$. Enumerate elements of $S$ as $\lambda_{i,R}$ and define
$$\xi(i,R)=\lambda_{i,R}(i,R)+R+1.$$
Notice that particular $\xi$ does not have any pair $\lambda\in S$, $R\in \mathbb{N}$ so that
$$\xi(i)-\lambda(i) \leq R$$
for all $i$.
\end{proof}

The aim of the remainder of this section is to offer a way of introducing two basic uniform structures on any quotient space of a subset $E$ of simple ends of a proper metric space $X$ (and this way can be easily generalized to arbitrary coarse space $X$ compatible with a locally compact topology on $X$). 

Let $q:E\to F$ be the quotient function on a subset $E$ of simple ends of a proper metric space $X$. Since $E\subset \bar X^{\mathbb{N}}$, where $\bar X$ is the Higson compactification of $X$, and $\bar X^{\mathbb{N}}$ is a compact space, it has the unique uniform structure inducing the topology on $\bar X^{\mathbb{N}}$. $E$ inherits that uniform structure. Now we have two options:\\
1. Declare $f:F\to M$, $M$ a compact metric space, to be uniformly continuous if and only if $f\circ q$ is uniformly continuous.\\
2. Declare $f:F\to M$, $M$ a metric space, to be uniformly continuous if and only if $f\circ q$ is uniformly continuous.

It is likely 1) is related to the uniform structure discussed in \cite{Har}.

\end{document}